\documentclass[11pt, a4paper]{amsart}

\makeatletter
\newcommand*{\rom}[1]{\expandafter\@slowromancap\romannumeral #1@}
\makeatother
\usepackage{graphicx,color}
\usepackage{psfrag}
\usepackage{amscd}
\usepackage{amsthm,amsmath,amssymb,graphics,hyperref}
\usepackage{amsmath,amssymb,hyperref}
\usepackage{amsfonts}
\usepackage{enumerate}
\usepackage{amsmath, amsthm}
\usepackage{amscd}
\usepackage[titletoc,title]{appendix}
\usepackage[T1]{fontenc}
\usepackage[utf8]{inputenc}
\input xy
\xyoption{all}

      \theoremstyle{plain}
      \newtheorem{theorem}{Theorem}[section]
      \newtheorem{lemma}[theorem]{Lemma}
      \newtheorem{corollary}[theorem]{Corollary}
      
      \theoremstyle{definition}
      \newtheorem{definition}[theorem]{Definition}

      \theoremstyle{remark}

\newcommand{\wc}{\mathfrak a^+}

      \makeatletter
      \def\@setcopyright{}
      \def\serieslogo@{}
      \makeatother

\begin{document}
    \author{Sungwoon Kim}
   \address{Department of Mathematics, Jeju National University, 102 Jejudaehak-ro, Jeju, 63243, Republic of Korea}
   \email{sungwoon@jejunu.ac.kr}

   \title[]{Limit sets for convex cocompact groups in higher rank symmetric spaces}

\begin{abstract}
We show that every limit point of a discrete subgroup $\Gamma$ of the isometry group of a symmetric space of noncompact type is conical if and only if $\Gamma$ is convex cocompact.
\end{abstract}




   \thanks{}
   \thanks{}

   \dedicatory{}

   \date{}

   \maketitle

\section{Introduction}
\label{}

Let $X$ be a symmetric space of noncompact type and $G$ be the identity component of the isometry group of $X$.
Let $\Gamma$ be a discrete subgroup of $G$. Choose a base point $o \in X$. Then the limit set $L_\Gamma$ of $\Gamma$ is defined by $$L_\Gamma= \overline{\Gamma \cdot o} \cap \partial_\infty X,$$ where $\partial_\infty X$ is the geometric boundary of $X$ with the cone topology.
In order to describe how the orbits of $\Gamma$ approach the limit points of $\Gamma$, 
there have been many studies of the types of limit points, such as conical limit points, radial limit points, and horospherical limit points (see \cite{Al,Li04,Li06-2}).
We concern ourselves with conical limit points, a type which has played an important role in studying convex cocompact groups in rank one symmetric spaces.

\begin{definition}
A point $\xi \in \partial_\infty X$ is called a \emph{conical limit point} of $\Gamma$ if for some (and hence every) geodesic ray $\phi_\xi$ tending to $\xi$, there exists a sequence $\{\gamma_n\} \subset \Gamma$ such that $d(\gamma_n x, \phi_\xi)$ is uniformly bounded.
\end{definition}

A group $\Gamma$ is said to be \emph{convex cocompact} if there is a $\Gamma$-invariant convex subset $C\subset X$ with compact quotient $C/\Gamma$. There are many important examples of convex cocompact groups in rank one symmetric spaces, such as quasi-Fuchsian surface groups and Schottky groups.
For higher rank symmetric spaces, it turns out that the notion of convex cocompact group does not give new examples beyond the rank one convex cocompact groups. More precisely, Kleiner--Leeb \cite{KL} and Quint \cite{Qui} proved that the only way to produce convex cocompact groups in higher rank symmetric spaces is to take products of uniform lattices and rank one convex cocompact groups.

For Hadamard spaces of pinched negative curvature, Bowditch \cite{Bo95} gave many equivalent definitions of a convex cocompact group. One of them is as follows.
When $X$ has negative sectional curvature, then $\Gamma$ is convex cocompact if and only if every limit point of $\Gamma$ is conical.
In higher rank symmetric spaces, it is still true that if $\Gamma$ is convex cocompact, then every limit point of $\Gamma$ is conical.
It is therefore of interest to explore whether the converse holds in higher rank symmetric spaces. 
The aim of this paper is to answer the question in the affirmative. Namely, the main theorem is the following.

\begin{theorem}\label{main}
Let $X$ be a symmetric space of noncompact type and let $\Gamma$ be a Zariski dense discrete subgroup of the isometry group of $X$. Then $\Gamma$ is convex cocompact if and only if  every limit point of $\Gamma$ is conical.
\end{theorem}

Link \cite{Li04} defined the notion of radially cocompact group (see Definition \ref{rc}) in higher rank symmetric spaces to replace the notion of rank one convex cocompact group and then showed that if $\Gamma$ is radially cocompact and Zariski dense, the Hausdorff dimension of the radial limit set in a given subset $\Gamma\cdot \xi$ equals the exponent of growth in the direction $\Gamma\cdot \xi$. The question we concern ourselves with is whether the notion of radially cocompact group gives a new class of discrete groups in higher rank symmetric spaces.

Examples of radially cocompact groups include convex cocompact groups in rank one symmetric spaces of noncompact type, uniform lattices acting on symmetric spaces, and products of rank one convex cocompact groups. As will be seen in Corollary \ref{cor}, these are essentially the only examples to produce radially cocompact groups.

\begin{corollary}\label{cor}
Let $X$ be a symmetric space of noncompact type and let $\Gamma$ be a Zariski dense discrete subgroup of the isometry group of $X$. If $\Gamma$ is radially cocompact, then $\Gamma$ is convex cocompact.
\end{corollary}

 Due to the rigidity result of Kleiner--Leeb \cite{KL} and Quint \cite{Qui}, Corollary \ref{cor} implies that the class of radially cocompact groups reduces to products of uniform lattices and rank one convex cocompact groups.

\section{Preliminaries}

Throughout this paper, let $X$ denote a symmetric space of noncompact type, $G$ the identity component of the isometry group of $X$, and $\Gamma$ a Zariski dense discrete subgroup of $G$.

\subsection{The geometric boundary}

The \emph{geometric boundary} (or ideal boundary) of $X$, denoted by $\partial_\infty X$, is defined as the set of equivalence classes of geodesic rays under the equivalence relation that two geodesic rays are equivalent if they are within a finite Hausdorff distance from each other. For any point $o\in X$ and any point at infinity $\xi \in \partial_\infty X$, there exists a unique unit speed geodesic ray $\phi_{o,\xi}$ emanating from $o$ and representing $\xi$. The pointed Hausdorff topology on geodesic rays emanating from $o\in X$ induces a topology on $\partial_\infty X$. This topology does not depend on the base point $o$ and is called the \emph{cone topology} on $\partial_\infty X$. 

Let $Y$ be a subset of $X$. Then the ideal boundary $\partial_\infty Y$ of $Y$ is defined as $$\partial_\infty Y= \overline Y \cap \partial_\infty X$$ where $\overline Y$ is the closure of $Y$ with respect to the cone topology on $X\cup \partial_\infty X$.

\subsection{Cartan decomposition}

Choose a base point $o \in X$. Let $K$ be the isotropy subgroup of $o$ in $G$. Denote by $\mathfrak g$ and $\mathfrak k$ the Lie algebras of $G$ and $K$, respectively. 
Then a Cartan involution induced from the geodesic symmetry at $o$ gives rise to a Cartan decomposition $\mathfrak g=\mathfrak k \oplus \mathfrak p$ into its $+1$ and $-1$ eigenspaces.
Let $\mathfrak a$ be a maximal abelian subspace of $\mathfrak p$ and $A=e^{\mathfrak a}$.
An image of $A\cdot o$ under the action of $G$ is called a \emph{maximal flat} in $X$.
Any maximal flat in $X$ is isometric to the Euclidean space whose dimension equals the rank of $X$.
A nonzero element $v \in \mathfrak p$ is said to be \emph{regular} if the centralizer of $v$ in $\mathfrak p$ is a maximal abelian subspace of $\mathfrak p$. Otherwise $v$ is said to be \emph{singular}.
It is well known that the set of regular elements of $\mathfrak a$ is the complement in $\mathfrak a$ of the union of a finite collection of hyperplanes in $\mathfrak a$.
A connected component of the set of regular elements of $\mathfrak a$ is called a \emph{Weyl chamber} in $\mathfrak a$.
Choose a Weyl chamber $\wc$ in $\mathfrak a$. Then it determines a Cartan decomposition $G=Ke^{\overline{\wc}}K$ where $\overline{\wc}$ is the closure of $\wc$ in $\mathfrak a$.
An image of $e^{\overline{\wc}}\cdot o$ under the action of $G$ is called a \emph{Weyl chamber} in $X$.

To each ordered pair of two points $(x,y) \in X\times X$, there is a unique vector $H(x,y)\in \overline{\wc}$ such that $x=go$ and $y=ge^{H(x,y)}o$ for some $g\in G$.
The vector $H(x,y)$ is called the \emph{Cartan vector} of the ordered pair $(x,y)$.
Let $\wc_1=\{ H \in \wc \ | \ \|H\|=1 \}$. For $k \in K$ and $H_1 \in \overline{\wc_1}$, denote by $[k,H_1]$ the unique class in $\partial_\infty X$ which contains the geodesic ray $\phi(t) = ke^{H_1t}o, \ t \geq 0$. Any point $\xi \in \partial_\infty X$ can be written in the form $[k, H_1]$; here, $k$ is called an \emph{angular projection}, and $H_1$ the \emph{Cartan projection} of $\xi$.
Note that the Cartan projection of $\xi$ is unique but its angular projection is unique up to right multiplication by an element in the centralizer of $H_1$ in $K$.
A point $\xi \in \partial_\infty X$ is said to be \emph{regular} if the Cartan projection of $\xi$ is regular. Otherwise $\xi$ is said to be \emph{singular}.
Let $\partial_\infty^\mathrm{reg}X$ be the set of regular points in $\partial_\infty X$. Then we have the natural projection $$\pi : \partial_\infty^\mathrm{reg}X \rightarrow K/M$$ defined by $\pi([k,H_1])=kM$, where 
$M$ is the centralizer of $\mathfrak a$ in $K$. The homogeneous space $K/M$ is called the \emph{}Furstenburg boundary of $X$, and is denoted by $\partial_FX$.

\subsection{Axial isometries}

For each element $g \in G$, its displacement function $d_g : X\rightarrow \mathbb R$ is defined by $d_g(x)=d(x,gx)$. If $d_g$ has a positive minimum value in $X$, then $g$ is said to be \emph{axial} and
the \emph{translation length} $l(g)$ of $g$ is defined by $$l(g)=\inf_{x \in X}d(x,gx).$$
From now on, $g$ is assumed to be axial. By Proposition I.2.3 in \cite{Par}, the limits of $g^no$ and $g^{-n}o$ exist independently of $o$. Hence the limit $g^+=\lim_{n\rightarrow \infty}g^no$ is called the \emph{attractive fixed point} of $g$ and the limit $g^-=\lim_{n\rightarrow \infty}g^{-n}o$ is called the \emph{repulsive fixed point} of $g$. 
Note that $g^+$ and $g^-$ are antipodal.
The set $$\mathrm{Ax}(g)=\{ x \in X \ | \ d(x, gx)=l(g) \}$$ is called the \emph{axis} of $g$. Then $Ax(g)$ is closed and consists of the union of all geodesics translated by $g$ (see \cite[Proposition 1.9.2]{Eb}). For a point $x \in \mathrm{Ax}(g)$, the Cartan vector $H(x, gx)$ is called the \emph{translation vector} of $g$ and denoted by $L(g)$. Note that the translation vector of $g$ is independent of the choice of $x \in \mathrm{Ax}(g)$.

\subsection{The Tits boundary}
The geometric boundary $\partial_\infty X$ of $X$ carries the angle metric $\angle_T$ defined as $$\angle_T(\xi, \eta)=\sup_{x\in X} \angle_x(\xi,\eta)$$ where $\angle_x(\xi,\eta)$ is the angle between the geodesic rays $\phi_{x,\xi}$ and $\phi_{x, \eta}$. The \emph{Tits boundary} of $X$, denoted by $\partial_TX$, is the metric space $(\partial_\infty X,\angle_T)$.
Two points $\xi, \eta \in \partial_\infty X$ are called \emph{antipodal} if $\angle_T(\xi,\eta)=\pi$.
Note that $\xi$ and $\eta$ are antipodal if and only if there exists a geodesic $\phi$ such that $\phi(\infty)=\xi$ and $\phi(-\infty)=\eta$.

Let $\mathcal N(M)$ be the normalizer of $\mathfrak a$ in $K$. Then $\mathcal W=\mathcal N(M)/M$ is called the \emph{Weyl group} of the pair $(\mathfrak g, \mathfrak a)$.
It is well known that $(\partial_\infty A, \mathcal W)$ is a spherical Coxeter complex which makes $\partial_\infty X$ into a spherical building, the so-called \emph{Tits building} associated with $X$ where $A=e^{\mathfrak a}$.
The ideal boundaries of maximal flats in $X$ exactly consist of the apartments with respect to the Tits building on $\partial_\infty X$.

\subsection{Limit set}

Recall that the limit set $L_\Gamma$ of $\Gamma$ is defined as $$L_\Gamma= \overline{\Gamma \cdot o} \cap \partial_\infty X.$$
It is obvious that $L_\Gamma$ is a $\Gamma$-invariant closed subset of $\partial_\infty X$.
The \emph{directional limit set} $P_\Gamma \subset \overline{\wc_1}$ is defined as the set of Cartan projections of limit points of $\Gamma$ and the \emph{limit cone} $\ell_\Gamma \subset \wc_1$ is defined as
$$\ell_\Gamma=\left\{ L(\gamma) / \| L(\gamma)\| \ | \ \gamma \in \Gamma \text{ regular axial} \right\}. $$ 
When $\Gamma$ is nonelementary, it turns out that $P_\Gamma =\overline{\ell_\Gamma}$ (see \cite{Be, Li06}).

Link \cite{Li06} showed that $K_\Gamma=\pi(L^\mathrm{reg}_\Gamma)$ is a minimal closed subset of $\partial_FX$ under the action of $\Gamma$ where $L^\mathrm{reg}_\Gamma$ is the set of all regular limit points of $\Gamma$.

\section{Proof of the main theorem}

We begin by proving the following lemma, which is a key observation for proving the main theorem.

\begin{lemma}\label{antipodal}
Let $X$ be a symmetric space of noncompact type and let $\Gamma$ be a Zariski dense discrete subgroup of $G$.
Suppose that every limit point of $\Gamma$ is conical. Then for any pair of antipodal regular points $\xi$ and $\hat{\xi}$ in $L_\Gamma$, $$\partial_\infty F(\xi,\hat{\xi}) \subset L_{\Gamma},$$
where $F(\xi,\hat{\xi})$ is the unique maximal flat in $X$ whose ideal boundary contains $\xi$ and $\hat{\xi}$.
\end{lemma}

\begin{proof}
We first prove the lemma for pairs consisting of the attractive fixed point and repulsive fixed point of a regular axial isometry in $\Gamma$. Then, by using an approximation argument, we will prove it for an arbitrary pair of antipodal limit points.

Let $P_\Gamma$ denote the directional limit set of $\Gamma$ and $\ell_\Gamma$ the limit cone of $\Gamma$. It is known that the interior of $P_\Gamma$ is nonempty and $P_\Gamma=\overline{\ell_\Gamma}$ when $\Gamma$ is Zariski dense.
Moreover the limit set of $\Gamma$ in any Weyl chamber at infinity, if nonempty, is naturally identified with the set of directions in $P_\Gamma$. For more details, see \cite{Be, Li06}.
Hence there is a regular axial isometry $\gamma_0\in \Gamma$ such that the translation direction of $\gamma_0$ is an interior point of $P_\Gamma$.
Since $\gamma_0$ is a regular axial isometry, the axis $\mathrm{Ax}(\gamma_0)$ of $\gamma_0$ is the maximal flat $F(\gamma^+_0,\gamma^-_0)$ whose ideal boundary contains $\gamma^+_0$ and $\gamma_0^-$.
Moreover, $\gamma_0$ acts on $F(\gamma^+_0,\gamma^-_0)$ as a translation in the direction of $\gamma^+_0$. Let $W$ be a Weyl chamber of $F(\gamma^+_0,\gamma^-_0)$ with $\gamma^+_0\in \partial_\infty W$. According to Benoist's theorem in \cite{Be}, $\partial_\infty W \cap L_\Gamma$ is naturally identified with $P_\Gamma$. Thus the interior of $\partial_\infty W \cap L_\Gamma$ is nonempty and moreover $\gamma^+_0$ is an interior point of $\partial_\infty W \cap L_\Gamma$.

Let $S^1$ be an arbitrary isometrically embedded one-dimensional sphere in $\partial_\infty F(\gamma^+_0,\gamma^-_0)$ passing through $\gamma^+_0$ and $\gamma^-_0$. Then there is a Euclidean plane $\mathbb E^2 \subset F(\gamma^+_0,\gamma^-_0)$ with $\partial_\infty \mathbb E^2=S^1$. We choose a Cartesian coordinate system on the Euclidean plane $\mathbb E^2$, $(x_1,x_2)$ for which $o=(0,0)$ and the direction of $\gamma^+_0$ is $(1,0)$.
Since $\gamma^+_0$ is an interior point of $L_\Gamma \cap S^1$, there is an open circular arc $I$ in $S^1$ with $\gamma^+_0\in I \subset L_\Gamma \cap S^1$. Hence it is possible to choose two limit points $\xi, \eta \in I$ sufficiently near to $\gamma^+_0$ so that $\gamma^+_0 \in [\xi, \eta] \subset I$ where $[\xi,\eta]$ denotes the circular arc on $S^1$ connecting $\xi$ and $\eta$. We may assume that $\xi$ is in the direction of the positive $x_2$-coordinate and $\eta$ is in the direction of the negative $x_2$-coordinate.
Since every limit point of $\Gamma$ is a conical limit point, there are sequences $(\alpha_n)$ and $(\beta_n)$ in $\Gamma$ and a constant $D>0$ such that 
\begin{equation} 
d(\alpha_n o, \phi_{o,\xi})<D \text{ and }  d(\beta_n o, \phi_{o,\eta})<D.
\end{equation}
Consider the nearest point projection $\overline{\alpha_n o}$ of $\alpha_n o$ to $\mathbb E^2$.
Then it is obvious that $d(\overline{\alpha_n o}, \phi_{o,\xi})<D$. Since $\gamma_0$ acts on $\mathbb E^2$ as a translation in the direction of $(1,0)$, the orbit of $\overline{\alpha_n o}$ under the action of the cyclic group $\langle \gamma_0\rangle$ lies on the straight line $\ell_n \subset \mathbb E^2$ parallel to the $x_1$-axis and passing through $\overline{\alpha_n o}$.
Furthermore, the orbit of $\overline{\alpha_n o}$ under the action of $\langle \gamma_0\rangle$ is $l(\gamma_0)$-dense in $\ell_n$. 
Since $\xi$ is in the direction of the positive $x_2$-coordinate, we can assume that every $\overline{\alpha_n o}$ has a positive $x_2$-coordinate and the $x_2$-coordinate of  $\overline{\alpha_n o}$ tends to infinity. Thus $\{\ell_n\}_{n\in\mathbb N}$ is a family of straight lines parallel to the $x_1$-axis and going off to infinity from the $x_1$-axis.
Then it can be easily seen that for any half-line $\ell_+$ in the direction of the positive $x_2$-axis and any $n \in \mathbb N$,
$$d(\ell_+, \langle \gamma_0\rangle \cdot \overline{\alpha_n o})<l(\gamma_0).$$
This implies that every point on $S^1$ with the direction of the positive $x_2$-axis is contained in $L_\Gamma$.

Noting that $\eta$ is in the direction of the negative $x_2$-coordinate, a similar argument to the above leads to the conclusion that every point on $S^1$ with the direction of the negative $x_2$-coordinate is also contained in $L_\Gamma$. Therefore $S^1 \subset L_\Gamma$ for any one-dimensional sphere $S^1 \subset \partial_\infty F(\gamma_0^+,\gamma_0^-)$ passing through $\gamma^+_0$ and $\gamma^-_0$. Since the union of all one-dimensional spheres passing through $\gamma^+_0$ and $\gamma^-_0$ is $\partial_\infty F(\gamma^+_0,\gamma^-_0)$, we conclude that 
$\partial_\infty F(\gamma^+_0,\gamma^-_0) \subset L_\Gamma$.

Before dealing with arbitrary regular axial isometries of $\Gamma$, note that $\partial_\infty F(\gamma^+,\gamma^-) \subset L_\Gamma$ implies $P_\Gamma=\overline{\mathfrak{a}_1^+}$.
Now let $\gamma$ be an arbitrary regular axial isometry of $\Gamma$. Let $W$ be a Weyl chamber with $\gamma^+ \in \partial_\infty W$. Then since $P_\Gamma=\overline{\mathfrak{a}^+_1}$, it follows that $\partial_\infty W \cap L_\Gamma=\partial_\infty W$. Hence $\gamma^+$ is an interior point of $\partial_\infty W \subset L_\Gamma$. Repeating the previous argument leads to $\partial_\infty F(\gamma^+,\gamma^-) \subset L_\Gamma$.

Eberlein \cite{Eb} introduced the notion of $\Gamma$-duality as follows: Two points $\xi$ and $\eta$ in $\partial_\infty X$  are \emph{$\Gamma$-dual} if there exists a sequence $\{\gamma_n\} \subset \Gamma$ such that 
$$\gamma_n o \rightarrow \xi \text{ and }\gamma_n^{-1}o \rightarrow \eta \text{ as } n\rightarrow \infty.$$
For more information, see \cite[Section 1.9]{Eb}. 
Let $T^1X$ be the set of all unit vectors in $X$. Eberlein defined the nonwandering set $\Omega(\Gamma) \subset T^1X$ of $\Gamma$ which is an analogue of the limit set of $\Gamma$ (see \cite[Definition 1.9.10]{Eb}).
 By Proposition 1.9.14 in \cite{Eb}, the nonwandering set $\Omega(\Gamma)$ of $\Gamma$ can be defined by
the set of all unit vectors $v \in T^1X$ such that $\phi_v(\infty)$ and $\phi_v(-\infty)$ lie in $L_\Gamma$ and are $\Gamma$-dual where $\phi_v : \mathbb R \rightarrow X$ is the geodesic whose initial vector is $v$.
Let $\xi$ and $\hat{\xi}$ be antipodal limit points in $L_\Gamma$.

\begin{lemma}
Let $v$ be a unit vector in $T^1X$ with $\xi=\phi_v(\infty)$ and $\hat \xi=\phi_v(-\infty)$. Then $v$ lies in the nonwandering set $\Omega(\Gamma)$ of $\Gamma$.
\end{lemma}
\begin{proof}
By \cite[Proposition 1.9.14]{Eb}, it is sufficient to show that $\xi$ is $\Gamma$-dual to $\hat \xi$.
Let $D(\xi)$ be the set of limit points of $L_\Gamma$ that are $\Gamma$-dual to $\xi$. 
Then $D(\xi)$ is a $\Gamma$-invariant closed subset of $L_\Gamma$ by Proposition 1.9.13 in \cite{Eb} and thus
$\pi(D(\xi))$ is also a $\Gamma$-invariant closed subset of the limit set $K_\Gamma$ considered as a subset of the Furstenberg boundary of $X$. Since $K_\Gamma$ is a minimal closed set under the action of $\Gamma$ by \cite[Theorem 1.1]{Li06}, $$\pi(D(\xi))=K_\Gamma.$$
This implies that the limit set of $\Gamma$ in any Weyl chamber at infinity, if nonempty, has a point that is $\Gamma$-dual to $\xi$. 

Let $\eta \in L_\Gamma$ be a limit point that is $\Gamma$-dual to $\xi$.
Then, by definition, there exists a sequence $\{\gamma_n\} \subset \Gamma$ such that $\gamma_n o \rightarrow \xi$ and $\gamma_n^{-1}o \rightarrow \eta$ as $n\rightarrow \infty$.
Then the direction of $H(o,\gamma_no)$ tends to the direction of $\xi$ and the direction of $H(o,\gamma_n^{-1}o)$ tends to the direction of $\eta$. Hence it can be easily seen that $\eta$ is antipodal to $\xi$.
Therefore any limit point $\hat \xi$ of $L_\Gamma$ that is antipodal to $\xi$ must be $\Gamma$-dual to $\xi$. 
\end{proof}

By \cite[Proposition 4.9]{Li06}, there exists a sequence of regular axial isometries $\{h_n\} \subset \Gamma$ such that $h_n^+$ converges to $\xi$ and $h_n^-$ converges to $\hat{\xi}$. Furthermore, $d(x, F(h_n^+, h_n^-))$ remains bounded as $n \rightarrow \infty$. Hence $F(h_n^+, h_n^-)$ converges to the maximal flat $F(\xi,\hat{\xi})$ in the Chabauty topology. As we have seen before, $\partial_\infty F(h_n^+,h_n^-) \subset L_\Gamma$ for all $n \in \mathbb N$. Since $L_\Gamma$ is closed in $\partial_\infty X$, it follows that $\partial_\infty F(\xi,\hat{\xi}) \subset L_\Gamma$, which completes the proof.
\end{proof}

\begin{proof}[Proof of Theorem \ref{main}]
We will prove that $L_\Gamma$ is a top dimensional subbuilding of the Tits building associated with $X$.
It is sufficient to prove that any two points in $L_\Gamma$ are contained in a common apartment in the Tits building.
Let $\xi$ and $\eta$ be any two regular limit points of $L_\Gamma$. 
Since $\Gamma$ is Zariski dense, $\eta$ has an antipode $\hat \eta \in L_\Gamma$. By Lemma \ref{antipodal}, the apartment $A(\eta,\hat \eta)$ containing $\eta$ and $\hat \eta$ is contained in $L_\Gamma$.

According to Lemma 3.10.2 in \cite{KL97}, there is a point $\hat \xi \in A(\eta, \hat \eta)$ such that $$\pi=d(\xi, \hat \xi)=d(\xi,\eta)+d(\eta,\hat \xi).$$
Hence there is an apartment $A_0$ containing $\xi, \hat \xi$ and $\eta$. Since $\xi$ and $\hat \xi$ are antipodal regular limit points of $\Gamma$, it follows from Lemma \ref{antipodal} that $A_0 \subset L_\Gamma$.
Therefore, any two regular points of $L_\Gamma$ lie in an apartment contained in $L_\Gamma$.

Let $\xi'$ and $\eta'$ be arbitrary points in $L_\Gamma$. Since $P_\Gamma=\overline{\wc_1}$, it is possible to choose regular limit points $\xi$ and $\eta$ of $\Gamma$ sufficiently near to $\xi'$ and $\eta'$ respectively so that $\xi$ and $\xi'$ are contained in one Weyl chamber at infinity and $\eta$ and $\eta'$ are contained in one Weyl chamber at infinity. 
Then, as seen before, there is an apartment $A_0' \subset L_\Gamma$ such that $\xi, \eta \in A_0'$. Since $\xi$ and $\xi'$ are contained in one Weyl chamber at infinity, it follows that $\xi, \xi' \in A_0'$. Similarly $\eta, \eta' \in A_0'$ and thus $\xi', \eta' \in A_0' \subset L_\Gamma$. Hence $L_\Gamma$ is a top dimensional subbuilding of the Tits boundary of $X$.

Let $X=X_1\times \cdots \times X_k$ be the de Rham decomposition of $X$.
Applying Theorem 3.1 in \cite{KL} to the top dimensional subbuilding $L_\Gamma$ of $\partial_{T}X$, the limit set $L_\Gamma$ splits as a join $L_\Gamma=B_1 \circ \cdots \circ B_k$ where $B_i=\partial_T X_i$ when $X_i$ has rank at least two, and $|B_i|=\infty$ for each $i$.
Let $C_i$ be the closed convex hull of $B_i$. If $X_i$ has rank at least two, then $C_i=X_i$. Then $C:=\prod_i C_i$ is a $\Gamma$-invariant closed convex subset of $X$ with $\partial_\infty C=L_\Gamma$.
It only remains to prove that $\Gamma$ acts on $C$ cocompactly.

The cocompactness of the $\Gamma$-action on $C$ can be shown by Bowditch's proof of Proposition 6.13 in \cite{Bo}. For the reader's convenience, we here give a proof.
Suppose that $C/\Gamma$ is not compact. Then there is a sequence $\{x_n\}$ in $C$ such that $d(x_n, \Gamma o) \rightarrow \infty$ as $n \rightarrow \infty$.
Since $d(x_n, \Gamma o)=d( \Gamma x_n, o)$, it follows that $d(\Gamma x_n, o) \rightarrow \infty$ as $n \rightarrow \infty$.
Let $y_n$ be the point on $\Gamma x_n$ nearest to $o$, that is, $$d(y_n, o)=d(\Gamma x_n, o)=d(\Gamma y_n, o)=d(y_n,\Gamma o).$$
Note that $y_n \in C$ due to the $\Gamma$-invariance of $C$. By passing to a subsequence, we may assume that $y_n$ converges to some limit point $\xi \in L_\Gamma$.
Since every limit point of $\Gamma$ is conical, there is a sequence $\{\gamma_n\}$ in $\Gamma$ such that $\gamma_n o$ remains within a bounded distance of the geodesic ray $\phi_{o, \xi }$.
Now we have two sequences $\{y_n\}$ and $\{\gamma_n o\}$ in $X$ tending to $\xi$. This means that $\angle_o(y_n,\xi)$ and $\angle_o(\gamma_n o,\xi)$ converge to $0$ as $n \rightarrow \infty$.
Furthermore, since $\gamma_n o$ remains within a bounded distance of the geodesic ray $\phi_{o, \xi }$, we have that for any $n \in \mathbb N$, $\angle_{\gamma_n o}(y_m, \xi)$ converges to $0$ as $m \rightarrow \infty$.
Hence there is a geodesic triangle with vertices $o, \gamma_p o$ and $y_q$ such that $\angle_{\gamma_po}(o, y_q) >\frac{\pi}{2}$. 
By the law of cosines in a symmetric space $X$ of noncompact type, 
$$ d(o, y_q)^2 \geq d(o, \gamma_p o)^2 + d(\gamma_p o, y_q)^2 - 2 d(o, \gamma_p o)d(\gamma_p o, y_q)\cos \angle_{\gamma_po}(o, y_q).$$
Since $\angle_{\gamma_po}(o, y_q) >\frac{\pi}{2}$, the above inequality implies that $$d(o, y_q) > d(\gamma_p o, y_q) \geq d(\Gamma o, y_q) =d(o, y_q)$$
which contradicts the choice of $y_q$. Therefore $\Gamma$ acts on $C$ cocompactly.
\end{proof}

\section{Radially cocompact groups}

To define the notion of a radially cocompact group, we begin by recalling the definition of a radial limit point.

\begin{definition}\label{rc2}
A point $\xi \in \partial_\infty X$ is called a \emph{radial limit point} of $\Gamma$ if there exists a sequence $\{\gamma_n\}$ of elements of $\Gamma$ such that $\gamma_n o$ converges to $\xi$ and remains within a bounded distance of the union of closed Weyl chambers with apex $o$ containing the geodesic ray $\phi_{o,\xi}$ starting from $o$ and pointing towards $\xi$. The set of radial limit points is denoted by $L_\Gamma^{rad}$.
\end{definition}

\begin{definition}\label{rc}
A nonelementary discrete group $\Gamma < G$ is called \emph{radially cocompact} if there exists a constant $c>0$ such that for any $\xi \in 
\Lambda^{rad}_\Gamma$ and for all $t>0$, there exists an element $\gamma \in \Gamma$ with $$d(\gamma o, \phi_{o,\xi}(t))<c.$$
\end{definition}

Note that the notion of a radial limit point is a weaker notion than that of a conical limit point.
Definition \ref{rc} means that every radial limit point of a radially cocompact group is conical.  
Indeed it is not difficult to prove that every limit point of a radially cocompact group is conical, as follows.

\begin{lemma}\label{lemma1}
Let $\Gamma$ be a radially cocompact, Zariski dense discrete subgroup of $G$. Then every limit point of $\Gamma$ is a conical limit point. 
\end{lemma}

\begin{proof} First of all, it is obvious that the fixed points of regular axial isometries of $\Gamma$ are all conical limit points. 
By \cite[Theorem 4.10]{Li06}, if $\Gamma$ is Zariski dense, the set of attractive fixed points of regular axial isometries of $\Gamma$ is a dense subset of the limit set $L_\Gamma$. Hence, given a limit point $\xi \in L_\Gamma$, there exists a sequence $\{\xi_n\}$ of attractive fixed points of regular axial isometries of $\Gamma$ such that $\xi_n$ converges to $\xi$ in the cone topology on $\partial_\infty X$. Since each $\xi_n$ is a radial limit point and $\Gamma$ is radially cocompact, then for each $n\in \mathbb N$ and each $t>0$ there exists an element $\gamma_{n,t} \in \Gamma$ such that $$d(\gamma_{n,t} \cdot o, \phi_{o,\xi_n}(t))<c,$$ for some constant $c>0$ depending only on $\Gamma$.
Noting that $\phi_{o,\xi_n}$ converges to $\phi_{o,\xi}$ in the Chabauty topology and $\phi_{o,\xi_n}(t)$ converges to $\phi_{o,\xi}(t)$ as $n\rightarrow \infty$, it immediately follows from the triangle inequality that for each $t>0$ there exists a number $n(t)>0$ such that $$d(\gamma_{n(t),t} \cdot o, \phi_{o,\xi}(t))<2c,$$
which implies that $\xi$ is a conical limit point.
\end{proof}

Combining Lemma \ref{lemma1} with Theorem \ref{main}, Corollary \ref{cor} immediately follows.

\section*{Acknowledgements} The author is indebted to an anonymous reviewer of an earlier paper for providing insightful comments and providing directions for additional research which has resulted in this paper. 
This work was supported by a research grant of Jeju National University in 2017.


\begin{thebibliography}{00}


\bibitem{Al} P. Albuquerque, `Patterson--Sullivan theory in higher rank symmetric spaces', {\em Geom. Funct. Anal. }9 (1999), 1--28.

\bibitem{Be}
Y. Benoist,
`Automorphismes des c{\^o}nes convexes',
{\em Invent. Math. }141 (2000), no. 1, 149--193.

\bibitem{Bo95}
B. H. Bowditch,
`Geometrical finiteness with variable negative curvature', 
{\em Duke Math. J.} 77 (1995), 229--274.

\bibitem{Bo}
B. H. Bowditch,
`Relatively hyperbolic groups',
 \textit{Internat. J. Algebra Comput}. 22 (2012), no. 3, 1250016.

\bibitem{Eb}
P. Eberlein,
\textit{Geometry of Non-Positively Curved Manifolds}, 
Chicago Lectures in Mathematics, Chicago Univ. Press, Chicago (1996)

\bibitem{KL97}
B. Kleiner and B. Leeb,
`Rigidity of quasi-isometries for symmetric spaces and Euclidean buildings', 
{\em Publ. Math., Inst. Hautes \'{E}tud. Sci. } 86 (1997), 115--197.

\bibitem{KL}
B. Kleiner and B. Leeb,
`Rigidity of invariant convex sets in symmetric spaces', 
{\em Invent. Math. } 163(3) (2006), 657--676.

\bibitem{Li04}
G. Link,
`Hausdorff dimension of limit sets of discrete subgroups of higher rank Lie groups',
{\em Geom. Funct. Anal. }14 (2004), 400--432.

\bibitem{Li06}
G. Link,
`Geometry and dynamics of discrete isometry groups of higher rank symmetric spaces',
{\em Geom. Dedicata }122 (2006), 51--75.

\bibitem{Li06-2}
G. Link,
`Ergodicity of generalised Patterson--Sullivan measures in higher rank symmetric spaces',
{\em Math. Z.} 254 (2006), 611--625.

\bibitem{Par}
A. Parreau,
`D\'{e}g\'{e}n\'{e}rescences de sous-groupes discrets de groupes de Lie semisimples et actions de groupes sur des immeubles affines. Th\`{e}se de doctorat, Orsay (2000).

\bibitem{Qui}
J. F. Quint, 
`Groupes convexes cocompacts en rang sup\'{e}rieur',
{\em Geom. Dedicata } 113 (2005), 1--19.

\end{thebibliography}
\end{document}